\documentclass[11pt,a4paper,twoside]{amsart}
\usepackage{amsfonts,amssymb,amscd,amsmath,amsthm, enumerate}
\usepackage{graphicx}
\title{Absolutely compatible pairs in a von Neumann algebra}

\author{Nabin K. Jana, Anil K. Karn}
\address{School of Mathematical Science, National Institute of Science Education and
Research, HBNI, Bhubaneswar, At \& Post - Jatni, PIN - 752050, India.}
\email{nabinjana@niser.ac.in; anilkarn@niser.ac.in}

\author[A.M. Peralta]{Antonio M. Peralta}
\address{Departamento de An{\'a}lisis Matem{\'a}tico, Facultad de
Ciencias, Universidad de Granada, 18071 Granada, Spain.}
\email{aperalta@ugr.es}



\date{}

\newtheorem{theorem}{Theorem}[section]
\newtheorem{lemma}[theorem]{\bf Lemma}
\newtheorem{proposition}[theorem]{\bf Proposition}
\newtheorem{corollary}[theorem]{\bf Corollary}
\newtheorem{definition}[theorem]{\bf Definition}

\newtheorem{remark}[theorem]{\bf Remark}

\DeclareMathOperator{\rank}{rank}
\DeclareMathOperator{\trace}{trace}

\usepackage[colorlinks, bookmarks=true]{hyperref}
\usepackage{color,graphicx,shortvrb}
\usepackage[active]{srcltx} 

\begin{document}

\begin{abstract} Let $a,b$ be elements in a unital C$^*$-algebra with $0\leq a,b\leq 1$. The element $a$ is absolutely compatible with $b$ if $$\vert a - b \vert + \vert 1 - a - b \vert = 1.$$ In this note we find some technical characterizations of absolutely compatible pairs in an arbitrary von Neumann algebra. These characterizations are applied to measure how close is a pair of absolute compatible positive elements in the closed unit ball from being orthogonal or commutative. In the case of 2 by 2 matrices the results offer a geometric interpretation in terms of an ellipsoid determined by one of the points. The conclusions for 2 by 2 matrices are also applied to describe absolutely compatible pairs of positive elements in the closed unit ball of $\mathbb{M}_n$.
\end{abstract}

\maketitle

\section{Introduction}

The relation being orthogonal is a central notion of study in the setting of function algebras and in the non-commutative framework of general C$^*$-algebras. Let us recall that elements $a$ and $b$ in a C$^*$-algebra $A$ are called orthogonal ($a\perp b$ in short) if $a b^* = b^* a =0$. Hermitian elements are orthogonal precisely when they have zero product. Following the standard notation, we shall write $|a| =(a^* a)^{\frac12}$ for the absolute value of $a$.\smallskip

Several attempts to establish a non-commutative version of the celebrated Kakutani's theorem \cite{Kaku41}, which characterizes those Banach lattices which are lattice isomorphic to the space $C(\Omega)$, of all continuous functions on a compact Hausdorff space $\Omega$, have been pursued in recent years (cf. \cite{Karn1, Karn2, Karn3}). As in many previous forerunners, like the representation theory, published by Stone in \cite{Stone1940}, which characterizes $C(\Omega)$ in terms of order and its ring properties, a non-commutative Kakutani's theorem will necessarily rely on the notions of orthogonality, absolute value and order. Some discoveries have been found within this non-commutative program, for example, it is shown in \cite[Proposition 4.9]{Karn2018} that if $a$ is an arbitrary positive element in the closed unit ball, $\mathcal{B}_A$, of a unital C$^*$-algebra $A$ and $p$ is a projection in $A$, then $\vert p - a \vert + \vert 1 - p - a \vert = 1$ if and only if $a$ and $p$ commute. Furthermore, two positive elements $a$ and $b$ in $\mathcal{B}_A$ are orthogonal if, and only if, $a+b\leq 1$ and $\vert a - b \vert + \vert 1 - a - b \vert = 1.$ The second condition gives rise to a strictly weaker notion than the usual orthogonality. Accordingly to the notation in \cite{Karn2018}, given two elements $a,b\in A$ with $0\leq a,b \leq 1,$ we shall say that $a$ is \emph{absolutely compatible} with $b$ ($a\triangle b$ in short) if $$\vert a - b \vert + \vert 1 - a - b \vert = 1.$$ Clearly, $a\perp b$ (with $0\leq a,b\leq 1$) implies that $a\triangle b$. It is shown in \cite[Proposition 4.7]{Karn2018} that $a$ is absolute compatible with $b$ (with $0\leq a,b\leq 1$) if and only if $2 a \circ b = a + b - \vert a - b \vert,$ where $a \circ b = \frac{1}{2} (a b + b a)$ is the usual Jordan product of $a$ and $b$. This notion is applied in \cite{Karn2018} to introduce a spectral theory for absolute order unit spaces satisfying some specific conditions, which generalizes the spectral theory in von Neumann algebras.\smallskip

In a recent contribution we extend the notion of absolute compatibility to pairs of elements in the closed unit ball of an arbitrary (unital) C$^*$-algebra $A$ via absolute values (see \cite{JanKarnPe2018}). In this case we introduce notions which are strictly weaker than range and domain orthogonality. Concretely, elements $a$ and $b$ in $\mathcal{B}_A$ are \emph{domain} (respectively, \emph{range}) \emph{absolutely compatible} {\rm(}$a\triangle_d b$, respectively, $a\triangle_r b$, in short{\rm)} if $|a|$ and $|b|$ (respectively, if $|a^*|$ and $|b^*|$) are absolutely compatible, that is, $\Big| |a| -|b| \Big| + \Big| 1-|a|-|b| \Big| =1$ (respectively, $\Big| |a^*| -|b^*| \Big| + \Big| 1-|a^*|-|b^*| \Big| =1$). Finally, $a$ and $b$ are called \emph{absolutely compatible} {\rm(}$a\triangle b$ in short{\rm)} if they are range and domain absolutely compatible.\smallskip

One of the main results in \cite{JanKarnPe2018} proves that every contractive linear operator $T$ between C$^*$-algebras preserving domain absolutely compatible elements {\rm(}i.e., $a\triangle_d b$ in $\mathcal{B}_A$ $\Rightarrow T(a)\triangle_d T(b)${\rm)} or range absolutely compatible elements {\rm(}i.e., $a\triangle_r b$ in $\mathcal{B}_A$ $\Rightarrow T(a)\triangle_r T(b)${\rm)} is a triple homomorphism. Furthermore, a contractive linear operator between two C$^*$-algebras preserves absolutely compatible elements {\rm(}i.e., $a\triangle b$ in $\mathcal{B}_A$ $\Rightarrow T(a)\triangle T(b)${\rm)} if, and only if, $T$ is a triple homomorphism. Having in mind the extensive literature on bounded linear operators between C$^*$-algebras preserving (domain and/or range) orthogonality (cf., for example, \cite{Wolff94,Wong2005,BurFerGarMarPe2008,LiuChouLiaoWong2018,LiuChouLiaoWong2018b}), the results in \cite{JanKarnPe2018} inaugurate a new line to explore in the framework of preservers.\smallskip

After characterizing triple homomorphisms as contractive linear operators between C$^*$-algebras preserving absolutely compatible elements, it seems natural to explore how close or how far is a pair of absolutely compatible elements to be orthogonality. This comparison is the natural step in order to measure similarities and differences with linear orthogonality preservers. Absolute compatibility is not a mere technical workmanlike extension of previous notions, and will certainly play a role in the theory of preservers.  This paper is aimed to throw some new light to our knowledge on absolutely compatible pairs of positive elements in the closed unit ball of a von Neumann algebra. The first main result (see Theorem \ref{characterization in vN}) is a technical characterization showing that two elements $a,b$ in a von Neumann algebra $M$ with $0\leq a,b\leq 1$, are absolutely compatible if, and only if, denoting by $p_1$ the range projection of  $a\circ b$ in $M$, then $a$ and $b$ have matrix representations,
	say $a = \left( \begin{array}{cc} a_{11} & a_{12} \\  a_{12}^{\ast} &
	a_{22} \end{array} \right),$ and $b = \left( \begin{array}{cc} b_{11} &
	b_{12} \\  b_{12}^{\ast} & b_{22} \end{array} \right)$ with respect to the set
	$\{ p_1, 1-p_1 = p_2\}$ {\rm(}i.e., $a_{ij} = p_i a p_j$ and $b_{ij} = p_i b p_j${\rm)} satisfying certain technical identities. This technical description admits some finer reformulations established in Theorems \ref{comp} and \ref{comp1}.\smallskip

The \emph{range projection} of a positive element $a$ in a von Neumann algebra $M$ will be denoted by $r(a)$.  The \emph{almost strict part} of $a$ is defined as the element $e(a) := a - s(a)$, where $s(a) = 1 - r(1 - a)$ is the support projection of $a$. A non-zero element $0\leq a \leq 1$ in $M$ will be called \emph{strict}, if $s(a) = 0=n(a)$, where $n(a) = 1-r(a)$. The projections $s(a), r(e(a)),$ and $n(a)$ are mutually orthogonal with $s(a) + r(e(a)) + n(a) = 1$, and every element in $M$ admits a matrix decomposition with respect to this system of projections. Furthermore, elements $0\leq a,b\leq 1$ are absolutely compatible if, and only if, there exist $0\leq b_1 \leq s(a)$, $0\leq b_2 \leq r(e(a)),$ and $0\leq b_3 \leq n(a)$ such that $b_2$ is absolutely compatible with $e(a)$ and $b = b_1 + b_2 + b_3$ (see Theorem \ref{comp}).\smallskip

These characterizations are subsequently applied to determine when a pair of absolutely compatible elements in a von Neumann algebra is a commuting pair. When particularized to the von Neumann algebra $\mathbb{M}_n,$ of all $n\times n$ matrices with complex entries, the conclusions offer some interesting geometric interpretations. Commuting pairs of absolutely compatible positive elements in $\mathcal{B}_{\mathbb{M}_2}$ are described in Proposition \ref{p AC pairs in M2}. All possible pairs of absolutely compatible positive elements in $\mathcal{B}_{\mathbb{M}_2},$ in which one of the elements is not strict are considered in Proposition \ref{non-strict}. It is also shown that non-commuting strict matrices $0\leq a, b\leq 1$ in $\mathbb{M}_2$ are absolutely compatible if, and only if, $\det (a)>0,\; \det(b)>0$; $\trace(a) =1= \trace(b)$ and $\det(a\circ b) =0$ (see  Theorem \ref{strict}). Finally, (strict) matrices of the form $a = \left(
            \begin{array}{cc}
              t & \alpha\\
              \bar{\alpha} & 1 - t
            \end{array},\right)$ and
$b = \left(
\begin{array}{cc}
x & y + i z \\
y - i z & 1 - x
\end{array} \right)$ with $x,t\in (0,1)\backslash \{\frac12\},$ $|\alpha|^2 <t(1-t)$ and $|y+i z|^2< x(1-x)$, are absolutely compatible if, and only if, the corresponding point $\widetilde{b} = (x,y,z)$ in $\mathbb{R}^3$ lies in the ellipsoid $$\mathcal{E}_a\setminus\{\widetilde{p_a}, \widetilde{p_a'}\} = \{ \overline{x}\in \mathbb{R}^3  : d_2 (\overline{x}, \widetilde{a}) +d_2( \overline{x}, \widetilde{a'}) =1 \},$$ where $d_2$ denotes the Euclidean distance in $\mathbb{R}^3$, $\widetilde{a} = (t,\Re\hbox{e}(\alpha ), \Im\hbox{m}(\alpha ))$, $\widetilde{a'} = (1-t,-\Re\hbox{e}(\alpha ), -\Im\hbox{m}(\alpha ))$, $\widetilde{p_a} = \frac{1}{\lambda_1-\lambda_2} \left( t-\lambda_2, \Re\hbox{e}(\alpha), \Im\hbox{m}(\alpha)\right)$, and $\widetilde{p_a'}= \frac{1}{\lambda_1-\lambda_2} \left(\lambda_1 - t, -\Re\hbox{e}(\alpha), -\Im\hbox{m}(\alpha)\right)$ (see Theorem \ref{ellipticity}).\smallskip

In our final result we prove that absolutely compatible pairs of positive elements in $\mathcal{B}_{\mathbb{M}_n}$ can be represented as orthogonal sums of $2\times 2$ matrices which are pairwise absolutely compatible (cf. Theorem \ref{strict2n}).

\section{Characterisation of positive absolute compatible elements}\label{sec: AC in vN+}

The notion of absolute compatibility was originally introduced in the setting of positive elements in the closed unit ball of a unital C$^*$-algebra (cf. \cite{Karn2018}). Among the results in the just quoted paper we can find the following interested characterization:

\begin{proposition}\label{p 00}{\rm\cite[Propositions 4.5 and 4.7]{Karn2018}}
	Let $A$ be a unital C$^{\ast}$-algebra and let $0\leq a, b \leq 1$ be elements in $A$. Then the following statements are equivalent:
	\begin{enumerate}[$(a)$]
		\item $a$ is absolutely compatible with $b$;
		\item $2 a \circ b = a + b - \vert a - b \vert$;
		\item $a \circ b, (1 - a) \circ (1 - b) \in A^+$ and $( a \circ b ) \Big( (1 - a) \circ (1 - b) \Big) = 0$;
		\item $a \circ (1 - b), (1 - a) \circ b \in A^+$ and $\Big( a \circ (1 - b) \Big) \Big( (1 - a) \circ b \Big) = 0$.
	\end{enumerate}
\end{proposition}

Throughout the remaining sections we shall focus on the set $[0,1]_{A}$ of all positive elements in the closed unit ball of a C$^*$-algebra $A$.\smallskip

Absolute compatibility in the commutative setting can be characterized in the following form: two positive functions $a,b$ in the closed unit ball of $C(K)$ are absolutely compatible if and only if $a(t) b(t) =0$ for all $t$ in $K$ with $a(t),b(t)\in [0,1)$. Inspired by this characterization, our first result is a matricial decomposition of a pair of an arbitrary couple of absolutely compatible positive elements in the closed unit ball of a von Neumann algebra. 

\begin{theorem}\label{characterization in vN} Let $M$ be a von Neumann algebra, and let $a, b $ be elements in $[0, 1]_{M}$. Then the following statements are equivalent:\begin{enumerate}[$(a)$] \item $a$ is absolutely compatible with $b$;
\item There exists a projection $p_1$ in $M$ so that $a$ and $b$ have matrix representations,
	say $a = \left( \begin{array}{cc} a_{11} & a_{12} \\  a_{12}^{\ast} &
	a_{22} \end{array} \right),$ and $b = \left( \begin{array}{cc} b_{11} &
	b_{12} \\  b_{12}^{\ast} & b_{22} \end{array} \right)$ with respect to the set
	$\{ p_1, 1-p_1 = p_2\}$ {\rm(}i.e., $a_{ij} = p_i a p_j$ and $b_{ij} = p_i b p_j${\rm)} satisfying:
	\begin{enumerate}[$(b.1)$]
\item $p_1$ is the range projection of $a\circ b$;		
\item $a_{12} + b_{12} = 0$;
		\item $|a_{12}^{\ast}|^2
		= (p_1 - a_{11}) (p_1 - b_{11})$ {\rm(}and hence $a_{11} b_{11} = b_{11} a_{11}${\rm)};
		\item $|a_{12}|^2 = a_{22} b_{22}= b_{22} a_{22}$;
		\item $a_{12} = a_{11} a_{12} + a_{12} a_{22} = b_{11} a_{12} +
		a_{12} b_{22}$.
	\end{enumerate}
\end{enumerate}
\end{theorem}

\begin{proof} Suppose first that $a$ is absolutely compatible with $b$. We deduce from Proposition \ref{p 00} \cite{Karn2018} that $a \circ b$ and $(1 - a) \circ (1 - b)$ are orthogonal elements in $[0, 1]_M,$ and thus
	\begin{eqnarray}\label{eq (1)}
	(a \circ b)^2 = (a \circ b) (a + b - 1) = (a + b - 1) (a \circ b),
	\end{eqnarray}
	which gives $r(a \circ b) (a + b - 1) = (a + b - 1) r(a \circ b)$, or equivalently,
	\begin{eqnarray}\label{eq (2)}
	p_1 (a + b) = (a + b) p_1,
	\end{eqnarray} where $p_1 =r(a \circ b)$ is the range projection of $a\circ b$ in $M$.\smallskip
	
We shall distinguish three cases.\smallskip
	
Case $1$. If $p_1 = 0$, we have $a \circ b = 0,$ and thus $a \perp b$, because in this case $0\leq (1 - a) \circ (1 - b) = 1-a-b + \circ b = 1-a-b $, and hence $0\leq a+b\leq 1,$ which combined with $a\triangle b$ gives $a\perp b$ (cf. or \cite[Proposition 4.1]{Karn2018}).\smallskip
	
Case $2$. If $p_1 = 1$, we have $(1 - a) \circ (1 - b) = 0,$ and by similar arguments to those given above we get $(1 - a) \perp (1 - b)$.\smallskip
	
Case $3$. Let us assume that $p_1 \not= 1, 0$. Let $a = \left(
	\begin{array}{cc} a_{11} & a_{12} \\  a_{12}^{\ast} & a_{22} \end{array}
	\right)$ and $b = \left( \begin{array}{cc} b_{11} & b_{12} \\
	b_{12}^{\ast} & b_{22} \end{array} \right)$ be the matrix representations with respect to $\{ p_1, 1-p_1 \}$.
By applying the matrix representation in \eqref{eq (2)} we get $a_{12} + b_{12} = 0.$\smallskip

Furthermore, if we write $a \circ b = \left( \begin{array}{cc} m_0 & 0 \\  0 & 0\end{array} \right)$ in our matrix representation, we deduce from \eqref{eq (1)} that $m_0^2 = m_0 ( a_{11} + b_{11} - p_1)$, or equivalently,
	$m_0 ( a_{11} + b_{11} - p_1 - m_0) = 0$. It follows that
	$p_1 ( a_{11} + b_{11} - p_1 - m_0) =r(m_0) ( a_{11} + b_{11} - p_1 - m_0) = 0$, that is,
	\begin{eqnarray}\label{eq 10}
	m_0 = a_{11} + b_{11} - p_1.
	\end{eqnarray}
	
	Since $a$ is absolutely compatible with $b$, by Proposition \ref{p 00}, we also have $a \circ (1 - b), (1 - a) \circ b \in [0, 1]_M$ with $( a \circ
	(1 - b) ) ( (1 - a) \circ b ) = 0$. That is, $(a - a\circ b)
	(b - a\circ b) = 0$. Again, by computing the multiplication with respect to their matricial representations and applying \eqref{eq 10}, we get
	\begin{eqnarray*}
	(p_1 - b_{11}) (p_1 - a_{11}) - a_{12} a_{12}^{\ast} &=& 0; \\
	- (p_1 - b_{11}) a_{12} + a_{12} b_{22} &=& 0; \\
	a_{12}^{\ast} (p_1 - a_{11}) - a_{22} a_{12}^{\ast} &=& 0; \\
	- a_{12}^{\ast} a_{12} + a_{22} b_{22} &=& 0.
	\end{eqnarray*} These identities prove the desired statements.\smallskip
	
	We assume next the existence of a projection $p_1\in M$ such that $a$ and $b$ enjoy a matrix representation with respect to $\{p_1,1-p_1\}$ satisfying $(b.1)$-to-$(b.5)$. In this case, by $(b.2)$, the matrix representations of $a + b$ and $a - b$ with respect to $\{ p_1, 1- p_1 \}$ are
	$$a + b = \left( \begin{array}{cc} a_{11} + b_{11} &
	0 \\  0 & a_{22} + b_{22} \end{array} \right)\hbox{ and } a - b = \left(
	\begin{array}{cc} a_{11} - b_{11} & 2 a_{12} \\  2 a_{12}^{\ast} & a_{22}
	- b_{22} \end{array} \right),$$ respectively.\smallskip
	
	Thus, the element $(a - b)^2$ writes in the form  $$\! \left(\!\! \begin{array}{cc} (a_{11} - b_{11})^2
	+ 4 a_{12} a_{12}^{\ast} & 2 ( (a_{11} - b_{11}) a_{12} + a_{12}
	(a_{22} - b_{22}) ) \\  2 ( (a_{11} - b_{11}) a_{12} + a_{12}
	(a_{22} - b_{22}) )^{\ast} & 4 a_{12}^{\ast} a_{12} + (a_{22} -
	b_{22})^2 \end{array}\! \!\right).$$
Now, applying $(b.5)$, $(b.3)$ and $(b.4)$ we deduce that $$2 ( (a_{11} - b_{11}) a_{12} + a_{12}
	(a_{22} - b_{22}) )=0, $$ $$(a_{11} - b_{11})^2
	+ 4 a_{12} a_{12}^{\ast} = (a_{11} - b_{11})^2
	+ 4 (p_1-a_{11}) (p_1-b_{11}) $$ $$= \left( (p_1 - a_{11}) + (p_1 - b_{11}) \right)^2,$$ and $$4 a_{12}^{\ast} a_{12} + (a_{22} -
	b_{22})^2 = 4 a_{22} b_{22} + (a_{22} -
	b_{22})^2 = (a_{22} + b_{22})^2.$$

We therefore have $$(a-b)^2 = \! \left(\!\! \begin{array}{cc} \left( (p_1 - a_{11}) + (p_1 - b_{11}) \right)^2 & 0 \\  0 & (a_{22} + b_{22})^2 \end{array}\! \!\right),$$ and
	$$\vert a - b \vert = \left( \begin{array}{cc} (p_1 - a_{11}) + (p_1 - b_{11}) & 0 \\  0 & a_{22} + b_{22} \end{array} \right).$$
	Next, again by $(b.3)$ and $(b.4)$ we also have
	\begin{eqnarray*}
		a b &=& \left( \begin{array}{cc} a_{11} b_{11}
			- a_{12} a_{12}^{\ast} &  - a_{11} a_{12} + a_{12} b_{22} \\
			a_{12}^{\ast} b_{11} - a_{22} a_{12}^{\ast} & - a_{12}^{\ast}
			a_{12} + a_{22} b_{22} \end{array} \right) \\
		&=& \left( \begin{array}{cc} a_{11} + b_{11} - p_1 & - a_{11}
			a_{12} + a_{12} b_{22} \\  a_{12}^{\ast} b_{11} - a_{22}
			a_{12}^{\ast}  & 0 \end{array} \right),
	\end{eqnarray*} and, by $(b.5)$, $ a \circ b = \left( \begin{array}{cc} a_{11} + b_{11} - p_1 & 0
	\\ 0 & 0 \end{array} \right)$, which shows that $2 (a\circ b) = a+b-|a-b|,$ and Proposition \ref{p 00} implies that $a$ is
	absolutely compatible with $b$.
\end{proof}

Clearly, any two orthogonal elements commute. We have already commented that the same conclusion is not true for absolutely compatible positive elements in general. 
In the subsequent discussion, we shall find some finer
characterizations for positive absolutely compatible pairs in a von Neumann algebra, and we shall isolate the
reason for which a pair of positive absolutely compatible elements in the closed unit ball may not commute.\smallskip

Let $M$ be a von Neumann algebra, and let $\mathcal{P}(M)$ denote the lattice of all projections in $M$. Given $0\leq a \leq 1$ in $M$, we shall denote by $s(a)$ the support projection of $a$, that is $s(a) = 1 - r(1 - a),$ where, as before, $r(a)$ is the range projection of $a$. It is known that $$s(a) = \sup \{ p \in \mathcal{P}(M): p \le a \} = \sup \{ p \in \mathcal{P}(M): p a = p = a p \},$$ and hence $s(a) a = s(a) = a s(a)$. Next, we set $e(a) := a - s(a)$, and we call it the \emph{almost strict part} of $a$.
Then $s(a) + e(a) = a \le r(a)$, and we have $s(a) \le r(a)$ and $e(a) \le
r(a) - s(a)$. Thus $r(e(a)) \le r(a) - s(a)$ so that $a = s(a) + e(a) \le s(a)
+ r(e(a)) \in \mathcal{P}(M)$. It trivially follows that $r(a) = s(a) + r(e(a))$. Finally, we
define $n(a) := 1 - r(a)$. Then $\{ s(a), r(e(a)), n(a) \}$ is a set of
mutually orthogonal projections in $M$ such that $s(a) + r(e(a)) + n(a)
= 1$. Therefore, every $x \in M$ has a unique $3 \times 3$ matrix
representation with respect to this system which we shall call the
matrix representation of $x$ with respect to $a$. In particular, the
matrix representation of $a$ with respect to its own system of projections is $\left(
\begin{array}{ccc} s(a)&0&0\\ 0&e(a)&0\\ 0&0&0 \end{array} \right)$.

\begin{definition}\label{def strict}
An element $0\leq a \leq 1$ in a von Neumann algebra $M$ will be called \emph{strict} if $s(a) = 0=n(a),$ equivalently, $a= e(a)$ and $r(a) = r(e(a))=1$.
\end{definition}

Every positive invertible element in the closed unit ball with zero support projection is clearly strict. The almost strict part of a projection is always zero. Every element $a \in [0, 1]$ with non-zero almost strict part can be written as a sum of a (possibly zero) projection and an element in $[0, 1]$ which is strict in some hereditary von Neumann subalgebra; namely $a = s(a) + e(a)$. We gather next some basic properties of strict elements.  

\begin{proposition}\label{strict} Let $0\leq a\leq 1$ be a strict element in a von Neumann algebra $M$.
	\begin{enumerate}[$(i)$]
		\item Then $r(1 - a) = 1$;
		\item If $0\leq  b\leq 1$ is absolutely compatible with $a$ such that $ab=ba$, then $b$ is a projection;
		\item If $0\leq b \leq 1$ is strict such that $a$ is absolutely compatible with $b$, then $a b \not= b a$;
		\item If $x \in M$ such that $x^{\ast} a^2 x = x^{\ast} x$, then $x = 0$.
	\end{enumerate}
\end{proposition}

\begin{proof}$(i)$ Let $p = r(1 - a)$. Since $(1-p)  a= 1-p,$ and $a$ is strict, we get $0\leq 1-p\leq s(a)=0$, and thus $p = 1$.\smallskip
	
$(ii)$ Since $a b = b a,$ and $a$ is absolutely compatible with $b$, Proposition \ref{p 00} implies that $\vert a - b \vert = a + b - 2 a b$. Squaring, we get	$$a^2 + b^2 - 2 a b = a^2 + b^2 + 2 a b + 4 a^2 b^2
	-4 a b (a + b).$$
	$$0= a b + a^2 b^2- a b (a + b).$$
	
Thus $a (1 - a) (b - b^2)= 0$. Since $a$ is strict, we have $r(1 - a) = 1=r(a),$ and hence $b = b^2$.\smallskip
	
$(iii)$ follows from $(ii)$. \smallskip
	
	$(iv)$ Let us take $x \in M$ with $x^{\ast} a^2 x = x^{\ast} x$. Since $xx^{\ast} a^2 xx^* = xx^{\ast} xx^*,$ we have $(xx^{\ast})^n a^2 (xx^*)^m = (xx^*)^{n+m},$ for all $n,m\in \mathbb{N}$. Consequently, $z a^2 w = z w$ for all $z,w$ in the von Neumann subalgebra generated by $xx^*$. Since the range projection $p=r(xx^*)$ lies in the latter von Neumann subalgebra, we deduce that $p a^2 p= p$. Thus $p(1-a^2)p = 0$ and consequently $(1-a^2)^{\frac{1}{2}} p =0 =(1-a^2)p =0$. Therefore, $a^2 p = p a^2 = p$. In particular, $p=pa = ap\leq s(a)=0$, because $a$ is strict. This proves that $x x^*=0$ or equivalently $x=0$.
\end{proof}

Now, we prove a characterization for commuting pairs of positive elements in a von Neumann algebra.

\begin{theorem}\label{comm}
	Let $M$ be a von Neumann algebra and let $0\leq a \leq 1$ in $M$. Then
	$0\leq b \leq 1$ commutes with $a$ if, and only if, there exist
	$0\leq b_1 \leq s(a)$, $0\leq b_2 \leq r(e(a))$ and $0\leq b_3 \leq n(a)$ such that
	$b_2$ commutes with $e(a),$ and $b = b_1 + b_2 + b_3$.
\end{theorem}

\begin{proof} We assume first that $a b = b a$. Consider the matrix representations of $a$ and $b$ with respect to $\{s(a),r(e(a)),n(a)\}$:
	$$a = \left( \begin{array}{ccc} s(a)&0&0\\ 0&e(a)&0\\ 0&0&0 \end{array} \right); \qquad
	b = \left( \begin{array}{ccc} b_{11}&b_{12}&b_{13}\\ b_{12}^{\ast}&b_{22}&b_{23}\\
	b_{13}^{\ast}&b_{23}^{\ast}&b_{33} \end{array} \right).$$
Since $a b = b a$, computing the matrix multiplications, we may conclude that
\begin{equation}\label{comm1} b_{12} = b_{12} e(a),\  e(a) b_{23} = 0,\ b_{13} = 0, \hbox{ and } e(a) b_{22} = b_{22} e(a).
\end{equation}


By the first equality in \eqref{comm1} we get $b_{12} e(a)^2 b_{12}^{\ast} = b_{12} b_{12}^{\ast}$. Since $e(a)$ is a strict element in the von Neumann subalgebra $r(e(a)) M r(e(a))$, and $b_{12} b_{12}^{\ast}$ lies in the latter subalgebra, we conclude from Proposition \ref{strict}$(iv)$ that $b_{12}^{\ast} = 0 =b_{12} $. Next, the second equality in \eqref{comm1} implies that $(r(e(a)) - e(a)) b_{23} = b_{23}$. Similar arguments to those applied before give $b_{23} = 0$. Thus $b = b_{11} + b_{22} + b_{33}$ with $b_{11}
	\in [0, s(a)]$, $b_{22} \in [0, r(e(a))],$ $b_{33} = [0, n(a)],$ and  $e(a) b_{22} = b_{22} e(a)$.\smallskip
	
	The reciprocal implication can be easily checked in a routine way.
\end{proof}

\begin{lemma}\label{l absolute compatibility in hereditary subalgebras} Let $0\leq a,b\leq 1$ in a von Neumann algebra $M$. Suppose $p$ is a projection in $M$. Then the following statements hold:
\begin{enumerate}[$(a)$]\item If $a,b\in p M p$, then $a\triangle b$ in $M$ if, and only if, $a\triangle b$ in $p Mp;$
\item If $a,b\in  p M p \oplus (1-p) M (1-p)$, then $a\triangle b$ in $M$ if, and only if, $p a p\triangle p b p$ in $p Mp$ and $(1-p) a (1-p) \triangle (1-p) b (1-p)$ in $(1-p) M(1-p)$.
\end{enumerate}
\end{lemma}

\begin{proof}$(a)$ Suppose $a,b\in p M p$. By orthogonality $ |a-b| + |1-a-b| = |a-b| + |(1-p) +p -a-b| = |a-b| + (1-p) + | p -a-b|$ with $|a-b|, | p -a-b|\in p M p$. Therefore, $1= |a-b| + |1-a-b|$ if, and only if, $1-p =|a-b| + | p -a-b|$.\smallskip

The proof of $(b)$ follows by similar arguments.
\end{proof}

If instead of considering commuting pairs of positive elements we study positive absolutely compatible pairs, we get the following.

\begin{theorem}\label{comp} Let $M$ be a von Neumann algebra and let $0\leq a \leq 1$ be an element in $M$. Then
	$0\leq b \leq 1$ in $M$ is absolutely compatible with $a$ if, and only if, there exist $0\leq b_1 \leq s(a)$, $0\leq b_2 \leq r(e(a)),$ and $0\leq b_3 \leq n(a)$ such that $b_2$ is absolutely compatible with $e(a)$ and $b = b_1 + b_2 + b_3$.
\end{theorem}

\begin{proof} Assume first that $a$ is absolutely compatible with $b.$ By Proposition \ref{p 00} we know that
	\begin{equation}\label{eq dagger} \vert a - b \vert = a + b - 2 a \circ b,
	\end{equation} with $a \circ b \ge 0$. Consider the matrix representations of $a$ and $b$ with respect to  $\{s(a),r(e(a)),n(a)\}$:
	$$a = \left( \begin{array}{ccc} s(a)&0&0\\ 0&e(a)&0\\
	0&0&0 \end{array} \right); \qquad b = \left( \begin{array}{ccc} b_{11}&b_{12}&b_{13}\\ b_{12}^{\ast}&b_{22}&b_{23}\\
	b_{13}^{\ast}&b_{23}^{\ast}&b_{33} \end{array} \right).$$
Now, by matrix multiplication, we get
	$$2 a \circ b = a b + b a = \left( \begin{array}{ccc} 2 b_{11}& b_{12} +
	b_{12} e(a)& b_{13}\\ e(a) b_{12}^{\ast} + b_{12}^{\ast}&2 e(a) \circ
	b_{22}& e(a) b_{23}\\ b_{13}^{\ast} & b_{23}^{\ast}  e(a)& 0 \end{array}
	\right).$$
	
	Since $a \circ b \ge 0$ (see Proposition \ref{p 00}), we conclude that $b_{13} = 0$ and $e(a) b_{23} = 0$, and then $b_{23} = r(e(a)) b_{23} = 0$. Now, by \eqref{eq dagger}, we get
	$$\vert a - b \vert = a + b - 2 a \circ b = \left( \begin{array}{ccc}
	s(a) - b_{11}& - b_{12} e(a)& 0\\ - e(a) b_{12}^{\ast}& e(a) + b_{22}-
	2 e(a) \circ b_{22}& 0\\ 0& 0& b_{33} \end{array} \right)$$
	which assures that $\vert a - b \vert^2 = ( x_{ij} )$ where
\begin{eqnarray*}
		x_{11} &=& (s(a) - b_{11})^2 + b_{12} e(a)^2 b_{12}^{\ast},\\
		x_{12} &=& - (s(a) - b_{11}) b_{12} e(a) - b_{12} e(a) (e(a) + b_{22} - e(a) \circ b_{22}),\\
		x_{13} &=& 0 \hspace{.8 ex} = \hspace{.8 ex} x_{23},\\
		x_{22} &=& e(a) b_{12}^{\ast} b_{12} e(a) + (e(a) + b_{22}- 2 e(a) \circ b_{22})^2,\\
		x_{33} &=& b_{33}^2
\end{eqnarray*}
and $x_{ji} = x_{ij}^{\ast}$. Similarly, we can show that $(a - b)^2 =	(y_{ij})$ where
\begin{eqnarray*}
y_{11} &=& (s(a) - b_{11})^2 + b_{12} b_{12}^{\ast},\\
y_{12} &=& - (s(a) - b_{11}) b_{12} - b_{12} (e(a) - b_{22}), \ \ \ \ \ \ \ \ \ \ \ \ \ \ \ \ \ \ \ \ \ \ \ \ \ \ \ \\
y_{13} &=& 0 \hspace{.8 ex} = \hspace{.8 ex} y_{23},\\
y_{22} &=& b_{12}^{\ast} b_{12} + (e(a) - b_{22})^2,\\
y_{33} &=& b_{33}^2
\end{eqnarray*}	and $y_{ji} = y_{ij}^{\ast}$.
	We apply now that $\vert a - b \vert^2 = (a - b)^2,$ and hence $x_{ij} = y_{ij}$ for $1 \le i, j \le 3$. The equation $x_{11} = y_{11}$ gives $b_{12} e(a)^2 b_{12}^{\ast} = b_{12}b_{12}^{\ast}$, and thus, by Proposition \ref{strict}$(iv)$, we have $b_{12}^{\ast} = 0$ so that $b_{12} = 0$. \smallskip
	
	Now, comparing $x_{22} = y_{22}$, we get $$(e(a) + b_{22}- 2 e(a) \circ b_{22})^2 = (e(a) - b_{22})^2.$$
	Furthermore, since $\vert a - b \vert \ge 0$, we have $e(a) + b_{22}- 2 e(a) \circ b_{22} \ge 0,$ and hence
	$$\vert e(a) - b_{22} \vert = e(a) + b_{22}- 2 e(a) \circ b_{22}.$$
	Thus $e(a)$ is absolutely compatible with $b_{22}$ (compare Proposition \ref{p 00}). We have therefore shown that
	$b = b_{11} + b_{22}+ b_{33},$ where $b_{11} \in [0, s(a)]$, $b_{22} \in [0, r(e(a))]$, $b_{33} \in [0, n(a)],$ and $b_{22}$ is absolutely compatible with $e(a)$.\smallskip
	
	Conversely, let us assume that there exist $b_1 \in [0, s(a)]$, $b_2 \in [0, r(e(a))]$ and $b_3 \in [0, n(a)]$ with $b_2$ absolutely compatible with $e(a)$ such that $b = b_1 + b_2 + b_3$. We shall show that $a$ is absolutely compatible with $b$.
	Having in mind that the projections in $\{s(a),r(e(a)),n(a)\}$ are mutually orthogonal, we deduce that
	\begin{eqnarray*}
		\vert a - b \vert &=& \vert s(a) + e(a) - b_1 - b_2 - b_3 \vert \\
		&=& \vert (s(a) - b_1) + (e(a) - b_2) + (0 - b_3) \vert \\
		&=& \vert s(a) - b_1 \vert + \vert e(a) - b_2 \vert + \vert 0 - b_3
		\vert \\
		&=& (s(a) - b_1) + \vert e(a) - b_2 \vert + b_3
	\end{eqnarray*}
	and
	\begin{eqnarray*}
		\vert 1- a - b \vert &=& \vert 1 - s(a) - e(a) - b_1 - b_2 - b_3 \vert \\
		&=& \vert (0 - b_1) + (r(e(a)) -e(a) - b_2) + (n(a) - b_3) \vert \\
		&=& \vert 0 - b_1 \vert + \vert r(e(a)) - e(a) - b_2 \vert +
		\vert n(a) - b_3 \vert \\
		&=& b_1 + \vert r(e(a)) - e(a) - b_2 \vert + (n(a) - b_3).
	\end{eqnarray*}
	Adding these terms, we get
	$$\vert a - b \vert + \vert 1 - a - b \vert = s(a) + \vert e(a) - b_2 \vert
	+ \vert r(e(a)) - e(a) - b_2 \vert + n(a).$$

Since $e(a)$ is absolutely compatible with $b_2$, Lemma \ref{l absolute compatibility in hereditary subalgebras}$(a)$ implies that $\vert e(a) - b_2 \vert
	+ \vert r(e(a)) - e(a) - b_2 \vert =  r(e(a))$, and thus $\vert a - b \vert + \vert 1 - a - b \vert =1$, as desired.
\end{proof}

\begin{remark}\label{1}{\rm
Let $M$ be a von Neumann algebra and let $a, b \in [0, 1]$. If $a$
is absolutely compatible with $b$, then every distinct pair of elements in the
set $\{ s(a), e(a), n(a), b_1, b_2, b_3 \}$ is absolutely compatible. It is also true that each element in the
set $\{ s(a), e(a), r(e(a)), n(a), a, a + n(a) \}$ is absolutely compatible with every
element in the set $\{ b_1, b_2, b_3, b_1 + b_2, b_1 + b_3, b_2 + b_3,	b \}$.}
\end{remark}

\begin{remark}{\rm Let $M$ be a von Neumann algebra, and let $a, b \in [0, 1]$. If $a \triangle b$ then by Theorem \ref{comp},  $b = b_1 + b_2 + b_3$ such that $b_1 \in [0, s(a)]$, $b_2 \in [0, r(e(a))],$ $b_3 \in [0, n(a)],$ and $b_2 \triangle e(a)$. If in addition, $b$ is strict then $1= r(b) = r(b_1) + r(b_2) + r(b_3)$. It follows that $r(b_1) = s(a)$, $r(b_2)=r(e(a))$ and $r(b_3) = n(a)$. Now $b_1 = 0$ if, and only if, $s(a)=0$; $b_2 = 0$ if, and only if, $e(a)=0$; and $b_3 = 0$ if, and only if, $n(a)=0$. Also if some $b_i \neq 0$ for some $1\leq i \leq 3,$ then it must be strict.}
\end{remark}

We recall that for a projection $p$ and $0\leq a \leq 1$ in a C$^*$-algebra $A$, we have $p a = a p$ if, and only if, $p$ is
absolutely compatible with $a$ (see \cite[Proposition 4.9]{Karn2018}). Thus, it follows from Proposition \ref{strict}$(ii)$ that the next
result is an assimilation of Theorems \ref{comm} and \ref{comp}.

\begin{corollary} Let $0\leq a, b \leq 1$ be elements in a von Neumann algebra $M$. Then $a b = b a$ and $a\triangle b$ hold if, and only if, the following statements hold:
	\begin{enumerate}[$(1)$]
		\item $b = b_1 + b_2 + b_3$ with $0\leq b_1 \leq s(a)$, $0\leq b_2 \leq r(e(a))$ and $0\leq b_3 \leq n(a)$;
		\item $b_2$ is a projection with $b_2 e(a) = e(a) b_2$.
	\end{enumerate}
\end{corollary}

We further sharpen the conclusion in Theorem \ref{comp}.

\begin{theorem}\label{comp1} Let $a$ and $b$ be elements in a von Neumann algebra $M$ such that $0\leq a, b \leq 1$. Then
	$a$ and $b$ are absolutely compatible if, and only if, there exist a set of mutually orthogonal projections
	$$\{ s(a), s(b_2), r(e(b_2)), n_1(b_2), n(a) \}$$
	satisfying
	$$s(a) + s(b_2) + r(e(b_2)) + n_1(b_2) + n(a) = 1,$$
	$0\leq b_1 \leq s(a)$, $0\leq a_1 \leq s(b_2)$, $0\leq a_2, e(b_2) \leq r(e(b_2))$, $0\leq a_3 \leq n_1(b_2)$ and
	$0\leq b_3 \leq n(a)$ with $a_2$ absolutely compatible with $e(b_2),$ $a = s(a) + a_1 + a_2 + a_3,$ and $b = b_1 + s(b_2) + e(b_2) + b_3$.
\end{theorem}

\begin{proof} First, let $a$ be absolutely compatible with $b$. By Theorem \ref{comp}, there exist $b_1 \in [0, s(a)]$, $b_2 \in [0, r(e(a))]$ and $b_3 \in [0, n(a)]$ with $e(a)$ absolutely compatible with $b_2$ such that $b = b_1 + b_2 + b_3$.\smallskip

Consider now the absolutely compatible pair $(b_2, e(a))$ in $[0, r(e(a))]$ in the von Neumann algebra $r(e(a)) M r(e(a))$ (cf. Lemma \ref{l absolute compatibility in hereditary subalgebras}). Put $n_1(b_2) = r(e(a)) -
	s(b_2) - r(e(b_2))$. Now, by Theorem \ref{comp}, there exist $a_1 \in
	[0, s(b_2)]$, $a_2 \in [0, r(e(b_2))]$ and $a_3 \in [0, n_1(b_2)]$ with
	$a_2$ absolutely compatible with $e(b_2)$ such that $e(a) = a_1 + a_2 + a_3$, and thus $a = s(a) + a_1 + a_2 + a_3$ and $b = b_1 + s(b_2) + e(b_2) +
	b_3$.\smallskip
	
	Conversely, assume that there exists a set of mutually orthogonal projections
	$$\{ s(a), s(b_2), r(e(b_2)), n_1(b_2), n(a) \}$$
	satisfying
	$$s(a) + s(b_2) + r(e(b_2)) + n_1(b_2) + n(a) = 1,$$ $b_1 \in [0, s(a)]$, $a_1 \in [0, s(b_2)]$,
	$a_2, e(b_2) \in [0, r(e(b_2))]$, $a_3 \in [0, n_1(b_2)]$ and
	$b_3 \in [0, n(a)]$ with $a_2$ absolutely compatible with $e(b_2),$
	$a = s(a) + a_1 + a_2 + a_3$ and $b = b_1 + s(b_2) + e(b_2) + b_3$.
	Under these hypothesis we have
	\begin{eqnarray*}
		\vert a - b \vert &=& \vert (s(a) + a_1 + a_2 + a_3) - (b_1 + s(b_2) + e(b_2) + b_3) \vert \\
		&=& \vert (s(a) - b_1) + (a_1 - s(b_2)) + (a_2 - e(b_2)) + a_3 - b_3
		\vert \\
		&=& \vert s(a) - b_1 \vert + \vert a_1 - s(b_2) \vert + \vert a_2 -
		e(b_2) \vert + \vert a_3 \vert + \vert - b_3 \vert \\
		&=& (s(a) - b_1) + (s(b_2) - a_2) + \vert a_2 -
		e(b_2) \vert +  a_3 + b_3
	\end{eqnarray*}
	and
	$$\vert 1 - a - b \vert = \vert 1 - (s(a) + a_1 + a_2 + a_3) - (b_1 +
	s(b_2) + e(b_2) + b_3) \vert $$
	$$= \vert - b_1 - a_1  + (r(e(b_2)) - a_2 - e(b)) + ( n_1(b_2) - a_3)
	+ (n(a) - b_3) \vert $$
	$$= \vert - b_1 \vert + \vert - a_1 \vert + \vert r(e(b_2)) - a_2 -
	e(b_2) \vert + \vert n_1(b_2) - a_3 \vert + \vert n(a) - b_3 \vert$$
	$$= b_1 +  a_1 + \vert r(e(b_2)) - a_2 - e(b_2) \vert + (n_1(b_2) - a_3) + (n(a) - b_3).$$
	
Having in mind that $a_2$ is absolutely compatible with $e(b_2)$, Lemma \ref{l absolute compatibility in hereditary subalgebras} proves that
	$$\vert a_2 - e(b_2) \vert + \vert r(e(b_2)) - a_2 - e(b_2) \vert =
	r(e(b_2)).$$
	Thus, by adding the last three equations, we get $\vert a - b \vert + \vert
	1 - a - b \vert = 1$, which concludes the proof.
\end{proof}	

Let us observe that some of the projections, and consequently, some of the corresponding elements in Theorem \ref{comp1} may be zero.

\begin{corollary}\label{c precise conditions for commutativity} Let $a$ and $b$ be elements in a von Neumann algebra $M$ such that $0\leq a, b \leq 1$. Suppose that $a$ is absolutely compatible with $b$. Then $a b = b a$ if, and only if, the element $e(b_2)$ given by Theorem \ref{comp1} is zero, or equivalently, $b_2$ is a projection.
\end{corollary}

\begin{proof} Following the constructions of Theorem \ref{comp1}, we deduce that $a b = b a$ if, and only if,
	$a_2 e(b_2) = e(b_2) a_2$.\smallskip
	
We also know that, if $e(b_2) = 0$, then trivially $a_2 e(b_2) = e(b_2) a_2$.\smallskip
Having in mind that $e(b_2)$ is strict in the von Neumann algebra $r(e(b_2)) M r(e(b_2))$ and $a_2 \in r(e(b_2)) M r(e(b_2))$, we deduce via Proposition \ref{strict}$(ii)$ that $a_2$ is a projection. Furthermore, since $a_2 \le e(a)$ and $s(e(a)) = 0$, we get $a_2 = 0$. This implies that $a= s(a)+a_1+a_3$ is orthogonal to $e(b_2)$, and hence $r(e(b_2))\le n(a) = 1-r(a)$. But Theorem \ref{comp1} also implies that $r(e(b_2)) n(a) = 0$, which implies that $r(e(b_2)) = 0,$ or equivalently, $e(b_2) = 0$.
\end{proof}

We resume our previous conclusion in a more schematic form.

\begin{remark}\label{r final vN}{\rm Let $M$ be a von Neumann algebra and let $0\leq a, b \leq 1$ such that $a\triangle b$.
	\begin{enumerate}[$(1)$]
		\item If $a$ and $b$ commute, then there exist a set of mutually
		orthogonal projections $\{  s(a), b_2, n_1(b_2), n(a) \}$ in $M$
		whose sum is $1$ such that $a$ and $b$ have the following
		matrix representations with respect to this system:
		$$a = \left(
          \begin{array}{cccc}
            s(a) & 0 & 0 &  0 \\
		0 & a_1 & 0 & 0 \\
		0 & 0 & a_3 & 0 \\
		0 & 0 & 0 & 0
          \end{array}
        \right)
 \quad \textrm{and} \quad
		b =  \left( \begin{array}{cccc}
		b_1 & 0 & 0 &  0 \\
		0 & b_2 & 0 & 0 \\
		0 & 0 & 0 & 0 \\
		0 & 0 & 0 & b_3 \end{array} \right).$$
		Some of these projections may be zero.
		\item If $a$ and $b$ do not commute, then there exist a set of
		mutually orthogonal projections $\{  s(a), s(b_2), q_1, q_2,
		n_1(b_2), n(a) \}$ in $M$ whose sum is $1$ such that $a$ and
		$b$ have the following matrix representations with respect to this
		system:
		$a =  \left( \begin{array}{cccccc}
		s(a) & 0 & 0 &  0 & 0 & 0 \\
		0 & a_1 & 0 & 0 & 0 & 0 \\
		0 & 0 & a_{11} & a_{12} & 0 & 0 \\
		0 & 0 & a_{12}^{\ast} & a_{22} & 0 & 0 \\
		0 & 0 & 0 & 0 & a_3 & 0 \\
		0 & 0 & 0 & 0 & 0 & 0 \end{array} \right)$ and $b =  \left( \begin{array}{cccccc}
		b_1 & 0 & 0 &  0 & 0 & 0 \\
		0 & s(b_2) & 0 & 0 & 0 & 0 \\
		0 & 0 & b_{11} & - a_{12} & 0 & 0 \\
		0 & 0 & - a_{12}^{\ast} & b_{22} & 0 & 0 \\
		0 & 0 & 0 & 0 & 0 & 0 \\
		0 & 0 & 0 & 0 & 0 & b_3 \end{array} \right).$
		Some of these projections may be zero. However, $q_1$ and
		$q_2$ can not be zero. Similarly, $a_{12} \not= 0$. Moreover,
		$a_{11}$ commutes with $b_{11}$ and $a_{22}$ commutes with
		$b_{22}$.
	\end{enumerate}}	
\end{remark}

\section{Absolute compatibility in the case of matrices}

In this section we shall particularize the main conclusions in section \ref{sec: AC in vN+} to the case of matrix algebras. Note that, in a general von Neumann algebra $M$, $0$ and $1$ are absolutely compatible with every element $a \in [0, 1]_{M}$. Thus, in order to describe absolutely compatible elements in $M$ it suffices to discuss the absolutely compatible pairs in $[0, 1]_M \setminus \{ 0, 1 \}$.\smallskip

\subsection{Absolute compatibility in $\mathbb{M}_2$} In this subsection, we discuss the case of 2 by 2 matrices due to its special importance. Commuting pairs of absolutely compatible 2 by 2 matrices are described in the next result.

\begin{proposition}\label{p AC pairs in M2} Let $0\leq a,b\leq 1$ in $\mathbb{M}_2$ such that $a\triangle b$ and $a b = ba $. Then the following statements hold:
\begin{enumerate}[$(a)$]\item If $a$ is strict then there exit $\alpha,\beta\in (0,1)$ in $\mathbb{R}$ and a minimal projection $p\in \mathbb{M}_2$ such that $b = p$ and $ a = \alpha p + \beta (1-p)$;
\item If $a$ is not strict then one of the next statements holds:
\begin{enumerate}[$(b.1)$]\item If $e(a) = 0$, then there exist $\alpha,\beta\in [0,1]$ in $\mathbb{R}$ and a minimal projection $p\in \mathbb{M}_2$ such that $\alpha+\beta >0$ and $a = p$ and $ b = \alpha p + \beta (1-p)$;
\item If $s(a),e(a) \neq 0$, then there exist $\alpha,\beta\in [0,1]$ in $\mathbb{R}$ and a minimal projection $p\in \mathbb{M}_2$ such that $a = p +\alpha (1-p)$ and $b = \beta p + (1-p)$ with $\alpha\neq 0$ or $a = p +\alpha (1-p)$ and $b = \beta p $ with $\alpha,\beta\neq 0$;
\item If $n(a),e(a) \neq 0$, then there exist $\alpha,\beta\in [0,1]$ in $\mathbb{R}$ and a minimal projection $p\in \mathbb{M}_2$ such that $a =\alpha p $ and $b = p + \beta (1-p)$ with $\alpha\neq 0$ or $a = \alpha  p $ and $b = \beta (1-p) $ with $\alpha,\beta\neq 0$;
\end{enumerate}
\end{enumerate}
\end{proposition}

\begin{proof} If we analyze each one of the cases, the statement is a straight consequence of the conclusions in Corollary \ref{c precise conditions for commutativity} or Remark \ref{r final vN}$(1)$, the details are left to the reader.
\end{proof}

In our next result we study absolutely compatible pairs in which one of the elements is not strict.

\begin{proposition}\label{non-strict} Let $0\lneqq a,b\lneqq 1$ in $\mathbb{M}_2$ be an absolutely compatible pair. If $a$ is not strict, then there exists a minimal projection $p$ in $\mathbb{M}_2$ such that one of the following three cases arises:
\begin{enumerate}[$(1)$]
\item $a = p$ and $b = \lambda p + \mu (1-p)$ for some $\lambda, \mu \in [0, 1]$ with $0< \lambda + \mu <2$;
\item $a = p + t (1-p)$ and $b = \lambda p $ for some $t \in [0, 1)$ and $\lambda \in (0, 1]$;
\item $a = p + t (1-p)$ and $b = \lambda p + (1-p)$ for some $t, \lambda \in [0, 1)$.
\end{enumerate}
\end{proposition}

\begin{proof} Since $0\neq a$ is not strict then $s(a)\neq 0,$ and thus there exists a minimal projection $p= s(a)$ in $\mathbb{M}_2$ such that $a = p +t (1-p)$ for some $t\in [0,1]$ with $1-p \leq n(a) + r(e(a))$ and $1-p$ minimal. It follows that $n(a) =0$ or $r(e(a))=0$.\smallskip

If $n(a)\neq 0$, then $a= p= s(a),$ and by applying Theorem \ref{comp1} (see also Theorem \ref{comp} or Remark \ref{r final vN}$(2)$) we deduce the existence of $\lambda, \mu \in [0, 1]$ with $0< \lambda + \mu <2$ such that $b = \lambda s(a) + \mu n(a)$.\smallskip

If $e(a)\neq 0$, then $1-p = r(e(a))$ and Theorem \ref{comp} implies that either statement $(2)$ or $(3)$ holds.
\end{proof}

We shall deal next with non-commuting pairs in $\mathbb{M}_2$. Let $0\leq a, b  \leq 1$ be a non-commuting, absolutely compatible pair in $\mathbb{M}_2$. In this case $ a b \neq 0$. The case in which $a$ or $b$ is not strict is treated in Proposition \ref{non-strict}. We can thus assume, without any loss of generality, that $0\lneqq a, b  \lneqq 1$ are strict, absolutely compatible, and $ ab \neq0$. Henceforth, given $a\in \mathbb{M}_n,$ the symbol $\trace(a)$ will denote the (non-normalized) trace of $a$.

\begin{theorem}\label{strict} Let $0\lneqq a,b\lneqq 1$ in $\mathbb{M}_2$. Then $a$ and $b$ are strict and absolutely compatible with $ a b \neq b a$ if, and only if, the following three properties hold:
\begin{enumerate}[$(1)$]
\item $\det (a)>0,\; \det(b)>0$;
\item $\trace(a) =1= \trace(b)$;
\item $\det(a\circ b) =0$.
\end{enumerate}
\end{theorem}

\begin{proof} By hypothesis $n(a) = n(b) = s(a) = s(b) =0$. Accordingly to the notation and conclusions in Theorem \ref{comp1}, we also know that $s(b_2) = 0 = n(b_2)$. Therefore, by the just quoted theorem (see also Remark \ref{r final vN}), we can assume, up to an appropriate representation, that $a= \left(
                                                                                                                          \begin{array}{cc}
                                                                                                                            \alpha_{11} & \alpha_{12} \\
                                                                                                                            \overline{\alpha_{12}} & \alpha_{22} \\
                                                                                                                          \end{array}
                                                                                                                        \right),$ and $b= \left(
                                                                                                                          \begin{array}{cc}
                                                                                                                            \beta_{11} & -\alpha_{12} \\
                                                                                                                            -\overline{\alpha_{12}} & \beta_{22} \\
                                                                                                                          \end{array}
                                                                                                                        \right),$ where $\alpha_{11}, \alpha_{22}\in \mathbb{R}_0^+,$ $\alpha_{12}\in \mathbb{C}$, and $\det (a), \det(b)>0$
(the latter because $n(a) = n(b) =0$). We also know from Theorem \ref{characterization in vN} that the following identities hold \begin{enumerate}[$(1)$]		
		\item $|\overline{\alpha_{12}}|^2
		= (1 - \alpha_{11}) (1 - \beta_{11})$;
		\item $|\alpha_{12}|^2 = \alpha_{22} \beta_{22}$;
		\item $\alpha_{12} = (\alpha_{11} + \alpha_{22} )\alpha_{12} = (\beta_{11} +
		 \beta_{22}) \alpha_{12}$.
	\end{enumerate}
The case $\alpha_{12} = 0$ is impossible because $a$ and $b$ do not commute. Therefore $ \alpha_{11} + \alpha_{22} = \beta_{11} +
		 \beta_{22} =1$. It is not hard to check that, in this case, $$a\circ b = \left(
                                                                             \begin{array}{cc}
                                                                               \alpha_{11} \beta_{11} - |\alpha_{12}| & 0 \\
                                                                               0 & \alpha_{22} \beta_{22} - |\alpha_{12}| \\
                                                                             \end{array}
                                                                           \right) 
= \left(
                                                                             \begin{array}{cc}
                                                                               \alpha_{11} \beta_{11} - |\alpha_{12}| & 0 \\
                                                                               0 & 0 \\
                                                                             \end{array}
                                                                           \right),$$ and thus $\det (a\circ b ) =0$.\smallskip

Now let us assume that $a$ and $b$ satisfy $(1)-(3)$. Since $a\circ b$ is a hermitian matrix with $\det (a \circ b) = 0$, then there exists a rank one projection $p \in \mathbb{M}_2$ such that $a\circ b = \lambda p$ for some $\lambda \in [0,1]$. Find a unitary $u$ such that $u^* p u = \left(
                                                                                                                                         \begin{array}{cc}
                                                                                                                                           1&0\\0&0
                                                                                                                                         \end{array}
                                                                                                                                       \right)
$. As $\trace(a) = 1 = \trace(b)$, we have $a_u = u^* a u =  \left(
                                                               \begin{array}{cc}
                                                                 t & \alpha\\\overline{\alpha}&1-t
                                                               \end{array}
                                                             \right)
$ and $b_u = u^* b u = \left(
                                                        \begin{array}{cc}
                                                          s&\beta\\
\overline{\beta}&1-s
                                                        \end{array}
                                                      \right)
$. Having in mind that $a_u \circ b_u=\left(
                        \begin{array}{cc}
                          \lambda&0\\0&0
                        \end{array}
                      \right)
,$ we get $\alpha + \beta = 0$ and $\alpha \overline{\beta} +\overline{\alpha}\beta +2 (1-t)(1-s) = 0$. Thus $\beta = -\alpha$ and $|\alpha|^2 =(1-t)(1-s)$. Theorem \ref{characterization in vN} implies that $a_u \triangle b_u$ and consequently $a \triangle b$.\smallskip
	
We shall next show that $a$ is strict. Clearly $n(a) =0$ because $\det(a)>0$. If $s(a) \neq 0$, then there exists a rank one projection $q \in \mathbb{M}_2$ such that $q \leq a$. In this case $a = q + t (1-q)$ for some $t\in [0,1]$. The condition $\trace (a) =1$ implies that $t=0$ and hence $\det(a) =0$, which is impossible. We can similarly prove that $b$ is strict.\smallskip

Finally, if $ a b  = b a $, then $a\circ b = a b = ba $ and $0 = \det (a\circ b) = \det( ab) = \det(a) \det(b) >0$, which is impossible. We have therefore shown that $a b\neq ba$, and in particular $ ab \neq 0$. 
\end{proof}

Let us comment some more concrete conclusions and geometric interpretations.

\begin{remark}{\rm $0\lneqq a,b\lneqq 1$ in $\mathbb{M}_2$ with $a\triangle b$ and $ a b \neq b a$. By Theorem \ref{strict}, there exists $t\in [0,1]$ and $\alpha\in \mathbb{C}$, such that $\det(a) = \det \left(
                                                               \begin{array}{cc}
                                                                 t & \alpha\\\overline{\alpha}&1-t
                                                               \end{array}
                                                             \right) = t(1-t)-|\alpha|^2\leq \frac{1}{4}-|\alpha|^2 \leq \frac14$. Similarly, $\det(b)\leq \frac14$.
Furthermore, $\det(a)=\frac{1}{4}$ if, and only if, $t=\frac{1}{2}$ and $\alpha=0$, or equivalently, $a=\frac{1}{2} 1$, which is impossible because $a b \neq b a$.  we have therefore shown that $\det(a), \det(b) <\frac{1}{4}$.\smallskip

We note that the matrices $a$ and $b$ belong to the set
\begin{align*}
\mathcal{S} &= \left\{ c \in \mathbb{M}_2 : 0\leq c\leq 1,\ \trace(c)=1,\  0<\det(c)<\frac{1}{4} \right\} \\
&= \left\{\left(
                                           \begin{array}{cc}
                                            t & \alpha \\
                                            \bar{\alpha} & 1 - t
                                           \end{array}
                                         \right) : t  \in (0, 1), \alpha \in \mathbb{C} ~\textrm{and} ~\vert \alpha\vert^2 < t (1 - t) \right\}\setminus \left\{\frac{1}{2} 1 \right\}.
\end{align*}
We observe that every element in $\mathcal{S}$ is strict.
}\end{remark}

Let $a$ be an element in $\mathcal{S}$. We can now conclude that, up to an appropriate $^*$-isomorphism, we can determine the set of absolutely compatible elements in $\mathcal{S}$.

\begin{theorem}\label{ellipticity}
Let $a$ and $b$ be two matrices in the set $\mathcal{S}\subset \mathbb{M}_2$. Then $a$ is absolutely compatible with $b$ if, and only if, there exists a unitary $u\in \mathbb{M}_2$ such that $u^* a u = \left(\begin{array}{cc}
t & \alpha \\
\bar{\alpha} & 1 - t
\end{array}
\right)$ and $u^* b u =\left(\begin{array}{cc}
 s & \beta \\
  \bar{\beta} & 1 - s
  \end{array}\right),$ where $\beta= -\alpha\neq 0$ and $s = \frac{|\alpha|^2}{t}$ or $s =1- \frac{|\alpha|^2}{1-t}$ with $s,t\in (0,1)$ and $|\alpha|^2 < t (1-t)$.
\end{theorem}

\begin{proof} We begin with some observations. The elements $a$ and $b$ are strict because they both lie in $\mathcal{S}$. Actually, $\det (a)>0,\; \det(b)>0$ and $\trace(a) =1= \trace(b)$.\smallskip

Suppose $a\triangle b$. Having in mind that for each unitary $u\in \mathbb{M}_2$ we have $u^* \mathcal{S} u = \mathcal{S}$, by applying Theorem \ref{characterization in vN} or Theorem \ref{comp1}, we can find a unitary element $u\in \mathbb{M}_2$ such that $a_u=u^* a u = \left(\begin{array}{cc}
t & \alpha \\
\bar{\alpha} & 1 - t
\end{array}
\right)\in \mathcal{S}$ and $b_u=u^* b u = \left(\begin{array}{cc}
s & -\alpha \\
-\bar{\alpha} & 1 - s
\end{array}
\right)\in \mathcal{S},$ with $|\alpha|^2 < \min\{ t (1-t) , s (1-s)\}$.\smallskip

In this case, $a_u b_u = b_u a_u$ if, and only if, $\alpha =0$ or $s +t =1$. The second case is impossible because if $s+t =1$ we would have $b_u = 1-a_u$ and thus $$ | b_u - a_u | + |1-b_u - a_u| = | b_u - a_u | = | 1 - 2 a_u| = \lambda_2 q_1 + |\lambda_1| q_2 \neq 1,$$ where $q_1$ is a minimal projection, $q_2 = 1-q_1$, and $$1>\lambda_2 = \sqrt{1 - 4 (t(1-t) -|\alpha|^2)}>0>\lambda_1 =  - \lambda_2 >-1$$ are the eigenvalues of $1 - 2 a_u$. The case $\alpha =0$ also is impossible because $s,t\in (0,1)\backslash\{\frac12\}$ and $a b= b a $.\smallskip

We have deduced that $a b \neq ba$ (equivalently, $a_u b_u = b_u a_u$). Theorem \ref{strict} implies that $a\triangle b$ (equivalently, $a_u\triangle b_u$) if, and only if, $\det (a\circ b) =0$ (equivalently, $\det (a_u\circ b_u) =0$). It can be easily seen that $\det (a_u\circ b_u) = (s t -|\alpha|^2)  ((1-s)(1-t) -|\alpha|^2)=0$, and thus $s t  = |\alpha|^2$ or $(1-s)(1-t) =|\alpha|^2$.\smallskip

On the other hand, since $|\alpha|^2 < t(1-t)<\frac14$, it can be easily seen that $t>|\alpha|^2$ or $(1-t)>|\alpha|^2$. If $t>|\alpha|^2$ and $(1-t)\leq |\alpha|^2$, we deduce that $s = \frac{|\alpha|^2}{t}\in (0,1)$ and $b_u = \left(\begin{array}{cc}
\frac{|\alpha|^2}{t} & -\alpha \\
-\bar{\alpha} & 1 - \frac{|\alpha|^2}{x}
\end{array}
\right)$. If $t\leq |\alpha|^2$ and $(1-t)> |\alpha|^2$, we deduce that $s =1- \frac{|\alpha|^2}{1-t}\in (0,1)$ and $b_u = \left(\begin{array}{cc}
1- \frac{|\alpha|^2}{1-t} & -\alpha \\
-\bar{\alpha} &  \frac{|\alpha|^2}{1-t}
\end{array}
\right).$ If $t>|\alpha|^2$ and $(1-t)>|\alpha|^2$ both solutions are possible.\smallskip

Finally, if there exists a unitary $u\in \mathbb{M}_2$ such that $u^* a u = \left(\begin{array}{cc}
t & \alpha \\
\bar{\alpha} & 1 - t
\end{array}
\right)$ and $u^* b u =\left(\begin{array}{cc}
 s & \beta \\
  \bar{\beta} & 1 - s
  \end{array}\right),$ in the conditions of the theorem, it is easy to check that $a b\neq ba$, $\det (a)>0,\; \det(b)>0$,
  $\trace(a) =1= \trace(b)$, and $\det(a\circ b) =0$. Theorem \ref{strict} assures that $a\triangle b$. 	
\end{proof}

The set $\mathcal{S}$ can be identified with the punctured open ball in $\mathbb{R}^3$ given by
\begin{align*}
\stackrel{\circ}{\mathcal{B}} &=\left\{(t, \Re\hbox{e}(\alpha), \Im\hbox{m}(\alpha)): 0<\left(t-\frac12\right)^2+\Re\hbox{e}(\alpha)^2+\Im\hbox{m}(\alpha)^2<\frac{1}{4}\right\}\\
& =\hbox{int}(\mathcal{B})\left(\left(\frac{1}{2},0,0\right),\frac{1}{2}\right)\setminus \left\{\left(\frac{1}{2},0,0\right)\right\}.
\end{align*}

Fix $a= \left(
                                           \begin{array}{cc}
                                            t & \alpha \\
                                            \bar{\alpha} & 1 - t
                                           \end{array}
                                         \right) \in \mathcal{S}$
and let us consider the corresponding point $\widetilde{a}=(t,\Re\hbox{e}(\alpha ), \Im\hbox{m}(\alpha ))$ in $\stackrel{\circ}{\mathcal{B}}$. Then $a':=1-a\in \mathcal{S}$ corresponds to the point $\widetilde{a'}=(1-t,-\Re\hbox{e}(\alpha ), -\Im\hbox{m}(\alpha ))\in \stackrel{\circ}{\mathcal{B}}$.\smallskip

Let $\lambda_1\geq \lambda_2 \geq 0$ be the eigenvalues of $a$. Then $1> \lambda_1 > \lambda_2>0$ with $\lambda_1 + \lambda_2=1$. Consider $p_a=\frac{1}{\lambda_1-\lambda_2} \left(
  \begin{array}{cc}
t-\lambda_2 &\alpha\\
\bar{\alpha}&\lambda_1 -t
  \end{array}
\right)
$. It is not hard to check that $p_a$ is a rank one projection, and the corresponding  point $\widetilde{p_a} = \frac{1}{\lambda_1-\lambda_2} \left( t-\lambda_2, \Re\hbox{e}(\alpha), \Im\hbox{m}(\alpha)\right)\in \mathbb{R}^3$ lies on the  outer boundary of $\stackrel{\circ}{\mathcal{B}}$ (that is, $\|\widetilde{p_a}-(1/2,0,0)\|_2^2 = 1/4$ in $\mathbb{R}^3$). Further, for the minimal projection $1 - p_a=p_{a}'=\frac{1}{\lambda_1-\lambda_2} \left(\begin{array}{cc}
                                               \lambda_1 - t &-\alpha\\ -\bar{\alpha} & t- \lambda_2
                                                                                    \end{array}
                                                                                  \right)
$, its corresponding point $\widetilde{p_a'}= \frac{1}{\lambda_1-\lambda_2} \left(\lambda_1 - t, -\Re\hbox{e}(\alpha), -\Im\hbox{m}(\alpha)\right)$ also lies on the outer boundary of $\stackrel{\circ}{\mathcal{B}}$.\smallskip

\begin{remark}\label{r ellipsoid}{\rm Let $a$ and $b$ be two matrices in the set $\mathcal{S}\subset \mathbb{M}_2$. Suppose $a$ is absolutely compatible with $b$. Suppose first that there exists a unitary $u\in \mathbb{M}_2$ such that $a_u=u^* a u = \left(\begin{array}{cc}
t & \alpha \\
\bar{\alpha} & 1 - t
\end{array}
\right)$ and $b_u =u^* b u =\left(\begin{array}{cc}
 \frac{|\alpha|^2}{t} & -\alpha \\
  -\bar{\alpha} & 1 - \frac{|\alpha|^2}{t}
  \end{array}\right),$ whose coordinates in $\mathbb{R}^3$ are $\widetilde{a_u}=(t,\Re\hbox{e}(\alpha ), \Im\hbox{m}(\alpha ))$ and $\widetilde{b_u}=(\frac{|\alpha|^2}{t},-\Re\hbox{e}(\alpha ), -\Im\hbox{m}(\alpha ))$.\smallskip

Consider the spheroid  $$\mathcal{E}_{a_u}=\{x \in \mathbb{R}^3: d_2(x, \widetilde{a_u}) + d_2(x, \widetilde{a_u'})=1 \}.$$  It is not hard to check that $$\|\widetilde{b_u}-\widetilde{a_u}\|_2 + \| \widetilde{b_u} - \widetilde{a_u'} \|_2 = \sqrt{\left(t-\frac{|\alpha|^2}{t}\right)^2 + 4 \Re\hbox{e}(\alpha )^2 +4 \Im\hbox{m}(\alpha )^2 } + \left| 1-t - \frac{|\alpha|^2}{t} \right|$$
$$ = \sqrt{\left(t-\frac{|\alpha|^2}{t}\right)^2 + 4 |\alpha|^2 } + \left| 1-t - \frac{|\alpha|^2}{t} \right| = \sqrt{\left(t+\frac{|\alpha|^2}{t}\right)^2 } + \left| 1-t - \frac{|\alpha|^2}{t} \right|$$ $$=\hbox{(since $t(1-t)>|\alpha|^2 $)} =  t+\frac{|\alpha|^2}{t} +  1-t - \frac{|\alpha|^2}{t} =1,$$ that is $\widetilde{b_u}\in \mathcal{E}_{a_u}$.  Similarly, for $b_u  =\left(\begin{array}{cc}
 1- \frac{|\alpha|^2}{1-t} & -\alpha \\
  -\bar{\alpha} &  \frac{|\alpha|^2}{1-t}
  \end{array}\right),$ we can also show that $\widetilde{b_u}\in \mathcal{E}_{a_u}$.
}\end{remark}

The geometric interest is that the conclusion in the above remark actually is a pattern for absolutely compatible elements in $\mathbb{M}_2$. Let us observe that $\widetilde{p_a}, \widetilde{p_a'}\in \mathcal{E}_{a}\backslash \mathcal{S}.$

\begin{theorem}\label{ellipticity}
Let $a = \left(
            \begin{array}{cc}
              t & \alpha\\
              \bar{\alpha} & 1 - t
            \end{array},\right)$ and
$b = \left(
\begin{array}{cc}
x & y + i z \\
y - i z & 1 - x
\end{array} \right)$ be two elements in $\mathcal{S}$. Then $a$ is absolutely compatible with $b$ if, and only if, the corresponding point $\widetilde{b} = (x,y,z)$ in $\mathbb{R}^3$ lies in $\mathcal{E}_a\setminus\{\widetilde{p_a}, \widetilde{p_a'}\}$.
\end{theorem}

\begin{proof} Let $|\alpha|= k$ and $\alpha= k e^{i\theta}$ for some $\theta \in[0, 2\pi]$. Assume that $b$ is absolutely compatible with $b$. Then by Theorem \ref{strict}, $\det(a \circ b)=0$. Let $\beta = y + i z$. To simplify notation, we set $\varrho = \Re\hbox{e}(e) (\alpha \overline{\beta}) = k (y\cos\theta +z \sin\theta ) $ By matrix calculations, we have  $\det(a \circ b) =0$ if, and only if, $$y^2 + z^2 +k^2 + 2 \varrho = | y + i z + k e^{i \theta}|^2 = 4 \left(t x + \varrho \right) \left((1-t)(1-x)+ \varrho \right)$$  which is equivalent to
\begin{align}\label{eq det a circ b}
& 4 t (1 - t) x^2 +  y^2 + z^2 -4 \varrho^2  +4 (1-2t )x \varrho \\ & - 4t (1 - t) x - 2(1-2t)\varrho +k^2= 0. \nonumber
\end{align}

On the other hand, the equation $$ \| \widetilde{a} -\widetilde{b}\|_2 + \| \widetilde{a'} -\widetilde{b}\|_2 = 1$$ can be rewritten in the form
\begin{align*}  \sqrt{(1-t-x)^2 + (-\Re\hbox{e}(\alpha)-y)^2 + (-\Im\hbox{m}(\alpha) - z)^2}  &\\
 = 1-\sqrt{(t-x)^2 + (\Re\hbox{e}(\alpha)-y)^2 + (\Im\hbox{m}(\alpha) - z)^2}. &
\end{align*} By squaring both sides and simplifying we get
$$ \sqrt{(t-x)^2 + (\Re\hbox{e}(\alpha)-y)^2 + (\Im\hbox{m}(\alpha) - z)^2} = t+x-2 t x - 2 \Re\hbox{e}(\alpha) y -2 \Im\hbox{m}(\alpha) z,$$ and by squaring one more time and simplifying we precisely arrive to \eqref{eq det a circ b}.\smallskip

We have proved that $a\triangle b$ if, and only if, $\det(a \circ b) =0$ if, and only if, \eqref{eq det a circ b} holds, if and only if, $\widetilde{b}\in \mathcal{E}_a\setminus\{p_a, p_a'\}.$\smallskip

An alternative approach can be obtained by substituting in \eqref{eq det a circ b} $Y=y\cos\theta + z\sin\theta = \frac{1}{k} \varrho$ and $Z=-y\sin\theta + z\cos\theta$, the previous equation transforms into
	$$4 t (1 - t) x^2+  Y^2 + Z^2 -4 k^2 Y^2 +4k(1-2t ) x Y - 4t (1 - t) x- 2k (1-2t) Y + k^2= 0,$$ which in turn reduces to
\begin{equation}\label{ddagger}  4d \left(U-\frac{1}{2} \cos\phi \right)^2 + \left(V- \frac{1}{2} \sin\phi \right)^2 +Z^2=d.
\end{equation}
Here $U= x \cos\phi -Y \sin\phi$, $V= x\sin\phi + Y \cos\phi$ with $\cos\phi=\frac{1-2 t}{\sqrt{1-4d}}$, $\sin\phi = \frac{2k}{\sqrt{1-4d}}$ and $d=\det(A)$. Hence \eqref{ddagger} and therefore \eqref{eq det a circ b} represent a prolate spheroid with semi-major axis $\frac{1}{2}$, eccentricity $\sqrt{1-4d}$. Further, we note that with respect to the co-ordinate system $(U,V,Z)$, \eqref{ddagger} has centre at $\left(\frac{1}{2} \cos\phi, \frac{1}{2} \sin\phi, 0\right)$; foci at $$F_1\left( \frac{1}{2}\left(\cos\phi - \sqrt{1-4d}, \frac{1}{2}\sin\phi, 0\right)\right)$$ and $$F_2\left( \frac{1}{2}\left(\cos\phi + \sqrt{1-4d}, \frac{1}{2}\sin\phi, 0\right)\right);$$ and the extremities of the major axis at $E_1\left( \frac{1}{2}\left(\cos\phi - 1, \frac{1}{2}\sin\phi, 0\right)\right)$ and $E_2\left( \frac{1}{2}\left(\cos\phi + 1, \frac{1}{2}\sin\phi, 0\right)\right)$.\smallskip
	
Transforming back to the coordinate system $(x,y,z)$, we may deduce that \eqref{eq det a circ b} has the centre at $\left(\frac{1}{2},0,0\right)$; foci at $\widetilde{a}$ and $\widetilde{a'}$ and the extremities of the major axis at $\widetilde{p_a}$ and $\widetilde{p_a'}$. 	
\end{proof}

\begin{remark}{\rm
	The director sphere of the prolate spheroid $\mathcal{E}_a$ (given by \eqref{eq det a circ b} in the proof of Theorem \ref{ellipticity}) is given by
	\begin{equation}\label{e20}
	x^2 + y^2 + z^2 = x
	\end{equation}
	which can be identified with $\mathcal{P}_1({\mathbb M_2})$ (:= the set of all rank one projection in $\mathbb M_2$) extending the identification between $\stackrel{\circ}{\mathcal{B}}$ and $\mathcal{S}$. Note that sphere given by \eqref{e20} is precisely the boundary of $\stackrel{\circ}{\mathcal{B}}$. Thus $\mathcal{P}_1({\mathbb M_2})$ may be called the (outer) boundary of $\mathcal{S}$. Similarly, the centre of (\ref{e20}) is $\left(\frac{1}{2}, 0, 0\right)$ which is identified with $\frac{1}{2} 1$. 
}\end{remark}

\subsection{Absolute compatibility in $\mathbb{M}_n$} Our next goal is to study absolutely compatible pairs in $\mathbb{M}_n$.

\begin{theorem}\label{strict2n}
Let $0\leq a, b\leq 1$ be strict elements in $\mathbb{M}_n$. If $a$ is absolutely compatible with $b$, then $n$ is an even number, say $n = 2 k,$ and there exist $a_1, \dots a_k$; $b_1, \dots b_k \in \mathcal{S}\subseteq \mathbb{M}_2$ and a unitary $w \in \mathbb{M}_n$ with $a_i \triangle b_i$ for every $i = 1, \dots k$ such that $a = w^* (a_1 \oplus \dots \oplus a_k) w$ and $b = w^* (b_1 \oplus \dots \oplus b_k) w$.
\end{theorem}

\begin{proof} Let $\rank(a \circ b) = k$. By Theorem \ref{characterization in vN}, for $p_1 = r(a\circ b)$, there exist $a_{11}, b_{11} \in \mathbb{M}_k = p_1 \mathbb{M}_n p_1$; $a_{22}, b_{22} \in \mathbb{M}_{n-k} = (1-p_1) \mathbb{M}_{n} (1-p_1)$ and $a_{12} \in \mathbb{M}_{k,n-k}$ with
$$a_{12} a_{12}^* = (1_k - a_{11})(1_k -b_{11}),\ \ a_{11} b_{11} = b_{11} a_{11},$$
$$a_{12}^*a_{12} = a_{22}a_{22} = b_{22} a_{22},$$
$$a_{12} = a_{11}a_{12} +a_{12}a_{22} = b_{11}a_{12} +a_{12}b_{22},$$
such that $a = \left(
                 \begin{array}{cc}
                   a_{11} & a_{12} \\
                   a_{12}^* & a_{22}
                 \end{array}
               \right)
$ and $b = \left(
             \begin{array}{cc}
               b_{11} & -b_{12} \\
               -b_{12}^* & b_{22}
             \end{array}
           \right)
$. Find unitaries $u_1 \in \mathbb{M}_k$ and $u_2 \in \mathbb{M}_{n-k}$ such that $u_1^* a_{11} u_1 = D_1$, $u_1^* b_{11} u_1 = E_1$, $u_2^* a_{22} u_2 = D_2$ and $u_2^* b_{22} u_2 = E_2$ are diagonal. Put $s_{12} = u_1^* a_{12} u_2$. Then for $u = u_1 \oplus u_2$,  we have
	$$a_u = u^* a u = \left(
                     \begin{array}{cc}
                       D_1 & s_{12}\\
                       s_{12}^* & D_2
                     \end{array}
                   \right)
 \mbox{ and } b_u = u^* b u = \left(
                                \begin{array}{cc}
                                  E_1 & -s_{12}\\
                                  -s_{12}^* & E_2
                                \end{array}
                              \right)
 $$ so that
\begin{equation}\label{E5} s_{12}s_{12}^* = (1_k - D_1)(1_k -E_1);\end{equation}
\begin{equation}\label{E6} s_{12}^* s_{12} = D_2 E_2;\end{equation}
\begin{equation}\label{E7} s_{12} = D_1 s_{12} + s_{12} D_2 = E_1 s_{12} +s_{12} E_2.\end{equation}

Let $D_1 = \mbox{diag}(\alpha_1,\ldots,\alpha_k)$, $D_2 = \mbox{diag}(\alpha_{k+1},\ldots,\alpha_n)$, $E_1 = \mbox{diag}(\beta_1,\ldots,\beta_k)$ and $E_2 = \mbox{diag}(\beta_{k+1},\ldots,\beta_n)$. Since $a$ and $b$ are strict so are $a_u$ and $b_u$. Thus $\alpha_i, \beta_j \in (0,1)$, for $1 \leq i, j \leq n$. In particular, $\rank (s_{12} s_{12}^*) = k$ and $ \rank (s_{12}^* s_{12}) =n-k$. Therefore, $$\max \{k, n-k\} \leq \rank(S_{12}) \leq \min\{k, n-k\}.$$
In other words, $\rank(s_{12}) = k = n-k$, so that $n=2k$.\smallskip
	
Let us denote $s_{12} = (s_{ij})$. Then by \eqref{E5} and \eqref{E6}, we have
\begin{equation}\label{E8}
\displaystyle\sum_{j=1}^k |s_{ij}|^2 = (1- \alpha_i)(1-\beta_i), \qquad 1\leq i \leq k
\end{equation}
\begin{equation}\label{E9}
\displaystyle\sum_{i=1}^k |s_{ij}|^2 =  \alpha_{k+j} \beta_{k+j}, \qquad 1\leq j \leq k
\end{equation}
\begin{equation}\label{E10}
\displaystyle\sum_{l=1}^k s_{il}\overline{s}_{jl} = 0, \qquad i \neq j
\end{equation}
\begin{equation}\label{E11}
\displaystyle\sum_{l=1}^k \overline{s}_{li}s_{lj} = 0, \qquad i \neq j.
\end{equation}

By \eqref{E7}, we have
\begin{equation}\label{E12}
s_{ij} = (\alpha_i + \alpha_{k+j}) s_{ij} = (\beta_i + \beta_{k+j}) s_{ij}, \qquad 1 \leq i, j \leq k.
\end{equation}

Since $\rank (s_{12}) = k$, $\det(s_{12} ) \neq 0$. Also by the Leibnitz formula for the determinant, we have $$\det(s_{12}) = \sum_{\sigma \in S_k} \mbox{sgn}(\sigma) \Pi_{i=1}^k s_{\sigma(i)i},$$ we note that $\Pi_{i=1}^k s_{\sigma(i)i} \neq 0$ for some permutation $\sigma \in S_k$. Thus replacing $s_{12}$ by $P_{\sigma^{-1}} s_{12}$, and $D_1$ by $P_{\sigma^{-1}} D_1 P_{\sigma}$, we may assume that $s_{ii} \neq 0$ for $1 \leq i \leq k$.\smallskip
	
	
Now, by \eqref{E12},
\begin{equation}\label{E13}
	\alpha_i + \alpha_{k+i} = 1 = \beta_i + \beta_{k+i} \qquad 1 \leq i \leq k.
\end{equation}

Applying \eqref{E13} to \eqref{E12}, we may deduce that
\begin{equation}\label{E14}
(\alpha_i - \alpha_j) s_{ij} = 0 = (\beta_i - \beta_j) s_{ij}, \qquad 1 \leq i, j \leq k.
\end{equation}
	
Applying \eqref{E13}, we further get that $D_1 + D_2 = 1_k = E_1 + E_2$. Thus by \eqref{E5} and \eqref{E6}, we have
	$$s_{12}s_{12}^* = s_{12}^* s_{12}.$$

Similarly, by \eqref{E7}, we deduce that
$$s_{12} D_1 = D_1 s_{12} \quad \mbox{and} \quad s_{12} E_1 = E_1 s_{12}.$$

Therefore, there exists a unitary $v \in \mathbb{M}_k$ such that $v^* D_1 v$, $v^* E_1 v$ and $v^* s_{12} v$ are diagonal, say,
\begin{equation*}
D := v^* D_1 v = \mbox{diag} (\lambda_1, \ldots \lambda_k)
\end{equation*}
\begin{equation*}
E := v^* E_1 v = \mbox{diag} (\mu_1, \ldots \mu_k)
\end{equation*}
\begin{equation*}
S := v^* s_{12} v = \mbox{diag} (s_1, \ldots s_k).
\end{equation*}
It follows that
$$a_0 := (v \oplus v)^* u^* a u (v \oplus v) = \left(
                                                 \begin{array}{cc}
                                                   D & S \\
                                                   S^*  & 1_k - D
                                                 \end{array}
                                               \right)
$$ and
$$B_0 := (v \oplus v)^* u^* b u (v \oplus v) = \left(
                                                 \begin{array}{cc}
                                                   E & -S \\
                                                   -S^* & 1_k - E
                                                 \end{array}
                                               \right)
$$ with $a_0 \triangle b_0$. We can find a suitable permutation $P \in \mathbb{M}_n$ such that
$$P a_0 P = \left(
                                            \begin{array}{cc}
                                              \lambda_1 & s_1 \\
                                              \bar{s}_1 & 1 - \lambda_1
                                            \end{array}
                                          \right)
 \oplus \ldots \oplus \left(
                        \begin{array}{cc}
                          \lambda_k & s_k \\
                          \bar{s}_k & 1 - \lambda_k
                        \end{array}
                      \right)
 $$ and
$$P b_0 P = \left(
              \begin{array}{cc}
                \mu_1 & - s_1 \\
                - \bar{s}_1 & 1 - \mu_1
                \end{array}
            \right)
 \oplus \ldots \oplus \left(
                        \begin{array}{cc}
                          \mu_k & - s_k \\
                          - \bar{s}_k & 1 - \mu_k
                        \end{array}
                      \right)
 $$ so that
$$\left(
    \begin{array}{cc}
      \lambda_i & s_i \\ \bar{s}_i & 1 - \lambda_i
    \end{array}
  \right)
 \triangle \left(
             \begin{array}{cc}
               \mu_i & - s_i \\
               - \bar{s}_i & 1 - \mu_i
             \end{array}
           \right)
  \quad \mbox{for} 1 \leq i \leq k.$$
\end{proof}

\textbf{Acknowledgements} Third author partially supported by the Spanish Ministry of Economy and Competitiveness (MINECO) and European Regional Development Fund project no. MTM2014-58984-P and Junta de Andaluc\'{\i}a grant FQM375.

\end{document}